\documentclass[11pt, final]{amsart}

\usepackage[cp850]{inputenc}

\usepackage{amssymb}

\usepackage[mathscr]{eucal}

\usepackage{amscd}

\newcommand{\R}{\mathbb R}

\newcommand{\N}{\mathbb N}

\theoremstyle{plain}

\newtheorem{theorem}{Theorem}

\newtheorem{prop}{Proposition}

\newtheorem{lemma}{Lemma}

\newtheorem{cor}{Corollary}

\theoremstyle{remark}

\title{Thick coverings for the unit ball of a Banach space}

\author{Jesus Castillo}

\address{Author's address: Departamento de Matematicas,
Universidad de Extremadura, 06071 Badajoz, Spain; e-mail: castillo@unex.es}

\author{Pier Luigi Papini}

\address{Author's address: Via Martucci, 19, 40136
Bologna, Italia; e-mail: plpapini@libero.it}

\author{Marilda Sim$\widetilde{o}$es}

\address{Author's address: Dipartimento di Matematica "G. Castelnuovo",
Universit\'a di Roma "La Sapienza", P.le A. Moro 2,
00185 Roma, Italia; e-mail: simoes@mat.uniroma1.it}

\thanks{This research has been supported
in part by project MTM2010-20190-C02-01 and the program Junta de
Extremadura GR10113 IV Plan Regional I+D+i, Ayudas a Grupos de
Investigaci\'on.}

\thanks{Keywords and phrases:  Thickness, Nets, Uniform non squareness}

\thanks{MR(2010) Subject Classification   46B20, 46B99}

\begin{document}

\maketitle

\begin{abstract}  We study the behaviour of Whitley's
thickness constant  of a Banach space with respect to
$\ell_p$-products and we compute it for classical
$L_p$-spaces.\end{abstract}

\section{Introduction and basic results}

This paper contains a study of Whitley's thickness constant and
its computation  in classical $L_p$ spaces  and $\ell_p$-products of Banach spaces.
Unless otherwise stated, we shall assume that $X$ is a real
infinite-dimensional Banach space, but  most results also hold in finite-dimensional spaces.
We shall denote by $B(x,r)$ the
ball centered at $x$, with radius $r$. The symbols $B_X$  and
$S_X$ will denote the unit ball and the unit sphere of $X$.
A finite set $F$ is said to be an
$\varepsilon$-$net$ for a subset $A \subset X$ if for any $a\in A$
there exists $f\in F$ such that $||a-f||\leq \varepsilon$.

Whitley introduced in \cite{W} the \emph{thickness} constant
$T_W(X)$ as follows:
$$T_W(X)= \inf\big\{\varepsilon>0: \, \mathrm{there \; exists \; an\;\;}
\varepsilon\mathrm{-net} \;\; F\subset S_X \mathrm{\; for} \;
S_X\}.$$ To study the thickness constant, it will be helpful to
consider the following equivalent formulation (see \cite[Prop.
3.4]{MP1}):
$$T(X)= \inf\big\{\varepsilon>0 : \; \exists
\{x_1,...,x_n\} \subset S_X  :  B_X\subset\bigcup_{i
\in\{1,...,n\}}B(x_i,\varepsilon) \big\}.$$

\begin{lemma}\label{uguale} If $X$ is an infinite dimensional Banach
space then $T(X) = T_W(X).$
\end{lemma}
\begin{proof} That $T_W(X)\leq T(X)$ is clear. The converse inequality follows
from \cite[Prop. 2]{CP2}, which we reproduce here for the sake of
completeness: \emph{Let $A$ be a subset of a Banach space $X$
which is weakly dense in its convex hull $conv(A)$. If a finite
family of convex closed sets covers $A$ they also cover the closed
convex hull of $A$.} Indeed, assume $A \subset\cup_{i
\in\{1,...,n\}} C_i$ for some closed convex sets $C_i$. Taking the
weak*-closures in $X^{**}$ one gets $\overline{conv}^{w^*}(A) =
\overline{A}^{w^*} \subset \cup_{i \in\{1,...,n\}}
\overline{C_i}^{w^*}$. Now, intersection with $X$ yields
$$\overline{conv}(A) = X \cap \overline{conv}^{w^*}(A) \subset
X\cap \cup_{i \in\{1,...,n\}} \overline{C_i}^{w^*}  = \cup_{i
\in\{1,...,n\}} C_i.$$
\end{proof}

This result can be considered a generalization (see \cite{papini})
of the antipodal theorem of Ljusternik and $\check{S}$nirel'man
(see \cite [p. 180]{LS} or else \cite{fv}): \emph{Let $X$ be an
infinite dimensional Banach space; if finitely many balls cover
the unit sphere of $X$, then at least one of them must contain an
antipodal pair $(y, -y)$}.

Note that if $X$ is any finite-dimensional space, then one has
$T(X)=1$, while $T_W(X)=0$ due to the compactness of $S_X$. It is
also clear that $T(X) \in [1,2]$ for every infinite-dimensional
space. Generalizations of $T(\cdot)$ were considered and studied
in \cite{MP1, CP2, D}; while relations with other parameters can
be seen in \cite{P, MP2, CP}. Spaces $X$ for which $T(X)=2$ have
been considered in \cite{BCP,KSW,L}). In particular, a Banach
space $X$ for which $T(X)=2$ must contain $\ell_1$ (\cite{BCP});
hence it cannot be reflexive (see also \cite[Thm. 1.2]{KSW}).
Thus, reflexive spaces $X$ have $T(X)<2$. Upper and lower
estimates for $T(X)$ in uniformly convex spaces, as well as upper
estimates in terms of the modulus of smoothness, follow from
results in \cite{MP2}. A reasonable characterization of the spaces
$X$ with $T(X)=1$ seems to be unknown. The value of $T(\cdot)$ in
many spaces is known (see \cite{P}); in particular: $T(c_0)=1$ and
$T(\ell_p)=2^{1/p}$ for $1 \leq p< \infty$. Our results in Section
3 can be considered the vector-valued generalization of these
estimates.

\section{Whitley constant of  $L_p$-spaces}

While it is known that $T(L_1)=2$ (see \cite[Ex. 3.6]{BCP1}), to
the best of our knowledge the thickness of $L_p[0,1]$ for $p>1$ is
unknown.

\begin{theorem} For  $1 \leq p < \infty$ one has $T(L_p[0,1])=2^{1/p}$.
\end{theorem}
\begin{proof}

Denote by $I$ the interval $[0,1]$. Let $\{f_1, ..., f_n\}$ be a
finite subset of $S_X$. Take $0<\varepsilon<1$. By the absolute
continuity of integrals, there exists $\sigma >0$ such that\\

(1) $\quad \int_A |f_i|^p< \varepsilon^p$ \, (for   $ i=1, . . . ,n$)     \,
whenever $\mu(A)<\delta$.\\

Take $A\subset I$ according to (1) and let  $f= \frac {\chi(A)}{
(\mu(A))^{1/p}} \; (f \in S_X)$. We have (for $ i=1, . . . ,n$): $
|| f-f_i ||^p    = \int_{I-A}|f_i|^p  + \int_A  \big| \chi(A) \,(
\frac{1}{(\mu(A))^{1/p}} -f_i) \big| ^p  \geq 1-   \int_A |f_i|^p
+   ||f -\chi(A) \, f_i ||^p    > 1- \varepsilon^p + \big| || f||
-  || \chi(A) \, f_i || \,  \big| ^p > 1- \varepsilon^p +
(1-\varepsilon)^p$. Since $\varepsilon>0$ is arbitrary, this shows
that $T(L_p[0,1]) \geq 2^{1/p}$.\\

Let $1 \leq p \leq 2$ and recall Clarkson's inequality:
$$||f+g||^q + ||f-g||^q \leq 2(||f||^p + ||g||^p)^{q/p}  \; \mathrm{where} \; 1/p + 1/q =1.$$
Taking  $f_0, f \in S_X$ one has
$$||f+f_0||^q + ||f-f_o||^q  \leq 2 (||f||^p + ||f_0||^p)^{q/p} = 2^{1+q/p} = 2^q,$$
and thus
$$\min  \{||f+f_0||,  ||f-f_0||\} \leq \Big(\frac {2^q}{2}\Big)^{1/q}= 2^{
\frac{q-1}{q}} = 2^{1/p},$$ so $T(X) \leq 2^{1/p}$ and the result
is proved for $1 \leq p \leq 2$.\\

Let now $2\leq p < \infty$. For $i=1, \dots ,n$ consider the norm
one functions $ \pm f_1, \dots,\pm f_n$ with $f_i=n^{1/p} \,
\chi_{[ \frac{i-1}{n},
 \frac{i}{n}] }$  . Take any $f \in S_X$; there exists $i$ such that   $\int_{
[\frac{i-1}{n},  \frac{i}{n}] }|f|^p  \leq   \frac{1}{n}$ (since
$\sum_{i=1}^n\int_{ [\frac{i-1}{n},   \frac{i}{n}] } |f|^p =1)$.
Denote by  $I_f=[\frac{i-1}{n},   \frac{i}{n}]$ the interval
corresponding to $f$. Recall Hanner's inequality (see \cite{H}):
for $p \geq 2$ one has
$$||f+g||^p + ||f-g||^p     \leq (||f||+||g||)^p +\big |
||f||-||g|| \big | ^p. $$

Apply this to the space $L_p(I_f)$: consider the restrictions of
$f$ and the $f_i$ to $I_f$, that we still denote in the same way,
to obtain
$$||f_i+f||^p + ||f_i-f||^p  \leq (||f_i|| + ||f||)^p +\big |
||f_i|| - ||f|| \, \big |^p \leq  \big(1+\frac{1}{n^{1/p} }\big)^p +   \big(1-
\frac{1}{n^{1/p} }\big)^p;$$
thus
$$\min \Big\{ \int_{I_f} |f_i+f|^p, \,
\int_{I_f} |f_i-f|^p \Big\} \leq  \frac{1}{2} \Big[\big(1+\frac{1}{n^{1/p} }\big)^p +   \big(1- \frac{1}{n^{1/p} }\big)^p \Big].$$

Therefore:
$$\min \Big\{ \int_I |f_i+f|^p, \,  \int_I |f_i-f|^p \Big\} \leq  \frac{1}{2} \Big[ \big(1+\frac{1}{n^{1/p} }\big)^p +   \big(1- \frac{1}{n^{1/p} }\big)^p \Big]+1;$$
 and then
$$\min  \big\{\min \{ ||f_i+f||, \, ||f_i-f|| \}: \,  i=1, . . . ,n  \big\} \leq \Big( \frac{1}{2}  \big[  (1+\frac{1}{n^{1/p} })^p +   (1- \frac{1}{n^{1/p} })^p \big] +1 \Big)^{1/p}.$$
Since we can take $n$ arbitrarily large, we obtain $T(L_p[0,1])
\leq 2^{1/p}$, which concludes the proof.
\end{proof}

\section {Whitley's constant in product spaces  }

Whitley's constant is strongly geometric, hence it is not strange that
thickness constants of  $X \oplus_p Y$ can be different for different values of $p$.
A bit more surprising is that the thickness constant of a product space $\ell_p(X_n)$
also depends on whether there is a  finite or infinite number of factors: indeed,
it follows from next theorem (part (1))  that  $T(c_0\oplus_2 c_0)=1$, while it follows
from Corollary 1 below  that $T(\ell_2(c_0))= \sqrt 2$.

\begin{theorem} \label{Sum 2} Let $1\leq p\leq \infty$.\begin{enumerate}

\item\label{uno} $T(X_1 \oplus_p \cdots \oplus_p X_N) \leq \max
\{T(X_n), \; 1\leq n\leq N \};$

\item\label{dos} $2^{1/p}\leq T(\ell_p(X_n)) \leq \left( \inf
\{T(X_n)\}^p+1\right )^{1/p}$, for $1\leq p<\infty$. The upper
estimate is also valid for finite sums.

\item\label{tres} $T(\ell_\infty(X_n)) = \inf \{T(X_n)\}$.
\end{enumerate}
\end{theorem}

\begin{proof} To prove (1), assume $p<\infty$; indeed,
 for $p=\infty$ it is contained in (3). Let us
call,  just for simplicity,  $Z= X_1 \oplus_p \cdots \oplus_p X_N$.
Given $\varepsilon>0$, let $\{x_1^n, ..., x_{k_n}^n\}$ ($n=1, . . . ,N)$ be a
$(T(X_n)+\varepsilon)-$net for $B_{X_n}$ with $\|x_i^n\|=1$. Take
in $(\R^N, \|\cdot\|_p)$ a finite $\varepsilon$-net
$\{(\lambda_1^k, \cdots, \lambda_N^k), \, 1\leq k \leq M\}$ for
its unit ball with $\sum_{j=1}^{N}|\lambda_j^k|^p=1$ for every
$k$. Consider all points $(\lambda_1^k x_{j_1}^1, \dots,
\lambda_N^k x_{j_N}^N)$ with $1\leq k \leq M$, $1\leq j_n\leq
k_n$, $1\leq n\leq N$. They form a finite subset of norm one
points of $Z$. Let us show they form a \, $\max \{T(X_n)+2\varepsilon:
\; 1\leq n\leq N \}$-net. Take $z=(z_1, \dots, z_N)\in B_Z$.
Choose an index $k(z)$ such that
$$\big\|(\|z_1\|,\dots, \|z_N\| ) - (\lambda_1^{k(z)}, \dots,
\lambda_N^{k(z)}) \big\|_p\leq\varepsilon.$$

Also, for each $n$ choose some index $i_n$ so that
$$\Big\|  \frac{z_n}{\|z_n\|} - x_{i_n}^n\Big\|_{X_n} \leq T(X_n)+\varepsilon.$$

Thus
\begin{eqnarray*} \big\|z - (\lambda_n^{k(z)}x_{i_n}^n)_n\big\|_Z &\leq& \big\|z
-(\|z_n\|x_{i_n}^n)_n\big\|_Z +  \big\| (\|z_n\|x_{i_n}^n)_n -
(\lambda_n^{k(z)}x_{i_n}^n)_n\big\|_Z\\ & \leq&
\big\|\|z_n\| \, (T(X_n)+\varepsilon)_n\big\|_p+\varepsilon\\
&\leq& \max \{T(X_n)\} + 2\varepsilon .\end{eqnarray*}

Since $\varepsilon>0$ is arbitrary, this proves (\ref{uno}).

\smallskip
The lower estimate in (\ref{dos}) is as follows. Let
$y_i=(y_i(n))_n$ be for $i=1, 2, ..., k$ a finite set of elements
of the unit sphere of $Y=\ell_p(X_n)$. Given $\varepsilon >0,$ let
$j$ be such that $||y_i(n)||_{X_i}< \varepsilon$ for $n\geq j $
and  all $i$. Take a norm one element $x\in X_j$ and form the
element $y=(0, \dots, 0, x, 0, \dots, 0))$ with $x$ at the j-th
position. One then has
\begin{eqnarray*}
\|y_i-y\|^p_Y &=& \|(y_i^{(1)}, ..., y_i^{(j-1)},y_i^{(j)}- x,
\;, y_i^{(j+1)}, ... )\|^p_Y\\
&=&\| (y_i^{(1)}\|^p_{X_1}+ ... + \|y_i^{(j-1)}\|^p_{X_{j-1}} +
\|y_i^{(j)}- x\|^p_{X_j} + \|(y_i^{(j+1)}\|^p_{X_{j+1}}+....\\
&=&1-\|y_i^{(j)}\|^p_{X_ j}+\|y_i^{(j)}- x\|^p_{X_j}\\
& >& 1-\varepsilon^p +|1-\varepsilon|^p.\end{eqnarray*}

This proves that $(T(\ell_p(X_n))^p\geq 1-\varepsilon^p
+|1-\varepsilon|^p$. Since $\varepsilon>0$ is arbitrary, the
result follows.

\smallskip
To obtain the upper estimate in (\ref{dos}), given $\varepsilon>0$ \, fix
$m$ and let $\{u_1, ..., u_t\}$ be a $(T(X_m)+\varepsilon)$-net
for $B_{X_m}$ with $||u_i||=1$ for all $1\leq i \leq t$. Consider
as a net for the unit ball of $Y=\ell_p(X_n)$ the points $v_i=
(0,\dots, u_i, \dots, 0)$, for $1\leq i\leq t$ ($u_i$ is in the $m$-th position). If $(x_n) \in
B_Y$ then in particular $||x_m||_{X_m}\leq 1$; fix
$i$ so that $\|x_m - u_i\|_{X_m} \leq T(X_m) +\varepsilon$. If
$p<\infty$, then $\|(x_n)- v_i\|_Y^p \leq \|x_m- u_i\|_Y^p+ 1
\leq(T(X_m)+\varepsilon)^p +1.$ This proves that
$T(\ell_p(X_n))^p\leq (T(X_m))^p+1$. Since $m$ is
arbitrary, the upper estimate follows.\\

The upper estimate in (\ref{tres}) is immediate from the arguments
above since when $p=+\infty$ one gets $T(\ell_p(X_n))\leq \max
\{T(X_m), 1\}=T(X_m)$. For the lower estimate, assume that
$T(\ell_\infty(X_n)))< \inf\{T(X_n)\}.$ Take $\varepsilon'>0$ and
$\alpha$ such that
$$T(\ell_\infty(X_n))<\alpha- \varepsilon' <\alpha < \alpha +
\varepsilon'< \inf\{T(X_n)\}$$ and fix $\varepsilon$ so that $
(1-\varepsilon)(\alpha + \varepsilon')> \alpha$. Take a finite
$\alpha$-net $\{z_1, \dots ,z_t\}$ for $B_{\ell_\infty(X_n)}$
verifying $\|z_i\|=1$ for each $i$. This in particular means that for each $i$,
given $\varepsilon>0$ there is some index $n_\varepsilon$ for
which $1-\varepsilon \leq ||z_i(n_\varepsilon)||_{X_ {n_\varepsilon}}\leq 1$. Set
$I_n(\varepsilon)=\{i: \, ||z_i(n)||_{X_ {n}}\geq 1-\varepsilon\}$. The
elements  $z_i(n)/\|z_i(n)\|, \, i\in I_n(\varepsilon)$ cannot
form an $(\alpha+ \varepsilon')$-net for $B_{X_n}$ and thus there must be  $x_n\in
B_{X_n}$ such that $||\frac{z_i(n)}{\|z_i(n)\|} -x_n||
> \alpha+\varepsilon'$ for all $i \in I_n(\varepsilon)$. Since $\bigcup_n
I_n(\varepsilon) = \{1,\dots, t\}$, for each $i\in \{1, \dots,
t\}$ there is some $n$ so that $i\in I_n(\varepsilon)$. Form (for each $i$) (one of) the
element(s) $x\in\ell_\infty(X_n)$ as $x(n)= \|z_i(n)\|x_n$ to get
the contradiction:

$$\| z_i - x\|_{\ell_\infty {(X_n)}} = \sup_n \left\{\|z_i(n)\| \left\|\frac{z_i(n)}{\|z_i(n)\|} -
x_n\right\|_{X_n} \right \} > (1-\varepsilon)(\alpha + \varepsilon')> \alpha.$$
 \end{proof}

As a consequence of (2) in the previous theorem we obtain.

\begin{cor} Let $X_n$ be a family of Banach spaces so that
$T(X_n)=1$ for at least one index $i$. Then $T(\ell_p(X_n)) =
2^{1/p}$.
\end{cor}

It is simple to see that estimates (1) and (2) (right inequality)
in  Theorem 2 are independent for $1< p<\infty$: for example,
consider a pair of spaces $X$, $Y$ with $T(X)=1$: if $T(Y)=1$,
then (1) is better; if $T(Y)=2$, then (2) is better. The upper
estimate in (2) is meaningful only if  \, min$\{T(X), T(Y)\}<
(2^p-1) ^{1/p}$. Both estimates are sharp: see Proposition 1 below
(where they coincide if we take $p=2$ and $T(Y)=1$; also, if $p=2$
and $T(Y)=3/2$, then we have strict inequality in both (1) and (2)
). According to Corollary 1, the estimate in (\ref{uno}) can fail
for infinite sums (also, we can observe that $T(\ell_p)=2^{1/p}$
and $T(\R)=1$).  The same corollary shows that, in general, one
can have $T(Y)>$ sup$\{T(X_n): n \in N\}$. Corollary 1 is not true
for the sum of two spaces: for example, according to (1) in
Theorem 2, $T(c_0 \oplus_1 c_0)=1$. This also shows that the lower
estimate in (2) of Theorem 2 does not apply in general to finite
sums. The same corollary shows that, in general, one can have
$T(Y)>$
sup$\{T(X_i): i \in N\}$.\\

The aim of the following example is twofold: first, it shows that
for $1<p<\infty$ one can have $T(X \oplus_p Y)>2^{1/p}$. Then, it
shows that it is possible to have $T(X\oplus_p Y)< \min \{(T(X),
T(Y)\}$.\\

\begin{lemma} $T(\ell_1\oplus_2 \ell_1)
=\sqrt{2+\sqrt2}<2$.\end{lemma}
\begin{proof} Let $Z=\ell_1\oplus_2 \ell_1$. Consider the first element, $e_1$,  of
the natural basis in $\ell_1$; take in $Z$ the four points
$z_1=(e_1,0); \; z_2=(-e_1,0); \; z_3=(0,e_1); \; z_4=(0,-e_1)$.
Let  $z=(x,y)\in S_Z;  \; ||x||_1=a; \; ||y||_1=b;\; a^2+b^2=1$;
this implies $1 \leq a+b \leq \sqrt2$. We want to prove that
$\min_{i=1,2,3,4} ||z-z_i||_Z \leq \sqrt{2+\sqrt2}$. One has
$$\min_{i=1,2} \|z-z_i\|^2_Z \leq (1+a)^2+b^2; \;\;\min_{i=3,4}
||z-z_i||^2_Z \leq a^2+(1+b)^2;$$ therefore $$\min_{i=1,2,3,4}
||z-z_i||^2_Z\leq \big((1+a)^2+ b^2+ a^2+(1+b)^2\big)/2 = 2+a+b$$
and thus
$$T(Z) \leq   \sup_{z \in S_Z} \,
 \min_{i=1,2,3,4} ||z-z_i||_Z\leq
\sqrt{2+a+b}\leq \sqrt{2+\sqrt2}.$$
Now assume that $(x_i, y_i),
..., (x_n, y_n)$ is a finite net of norm one elements for $B_Z$.
Given $\varepsilon>0$, there exists $k$  large enough such that
all sequences $(x_i), (y_i)$ have the $k^{th}$ component, in
modulus, smaller than or equal to $\varepsilon$. Take $(x,y)\in
S_Z$, such that both $x$ and $y$  have all components equal to
$0$, except the $k^{th}$ component  equal to $1/\sqrt2$. One has
that for all $i$:
\begin{eqnarray*}||(x,y)-(x_i,y_i)||_Z^2&=& ||x-x_i||_1^2+||y-y_i||_1^2 \\
&=& (a-|x_k|+|\frac{1}{\sqrt2}-x_k|)^2+
(b-|y_k|+|\frac{1}{\sqrt2}-y_k|)^2\\ &\geq&
(a-\varepsilon+\frac{1}{\sqrt2}-\varepsilon)^2 +
(b-\varepsilon+\frac{1}{\sqrt2}-\varepsilon)^2.\end{eqnarray*}

Since $\varepsilon$ is arbitrary, this proves that $$T(Z)\geq
\sqrt{ a^2+b^2+1+ \sqrt2(a+b)}
 \geq \sqrt{2+\sqrt2} \, ,$$
and the assertion follows.
\end{proof}
\begin{prop}

 Let    $Z = \ell_p \oplus_p Y$,  $1 \leq p < \infty$;  then  $T(Z) =
2^{1/p}$.
\end{prop}

\begin{proof}
Set  $X=\ell_p$; let $z=(x,y)\in B_Z:  x\in\ell_p;\, y\in Y; \, a^p+b^p=1$
where $ a^p=||x||_X^p; \;  b^p=||y||_Y^p$.
Consider the net given by the two points in $S_Z: z_1=(e_1,0); \,
z_2=-z$. In $\ell_p$ we have either $||e_1-x||_X^p\leq 1+a^p$ or
$||-e_1-x||_X^p\leq1+a^p$. Thus $||z_i-z||_Z^p=||\pm e_1-x||_X^p+
||y||_Y^p \leq 1+a^p+b^p$ \,for either $i=1$ or $i=2$.   This
proves that $T(Z)\leq 2^{1/p}$.

Let $z_1=(x_1, y_1), . . . ,
z_n=(x_n, y_n)$ be  a net for $S_Z$ from $S_Z$:  we have $
(||x_i||_X)^p + (||y_i||_Y) ^p=1$ for all $i$. Given
$\varepsilon>0$,  there is an index $j$ such that
$|(x_i)_j|<\varepsilon$ for $i=1, . . . , n$.  Consider the point
$z_j=(e_j, 0) \in Z$. Then we have, for every $i:
||z_i-z_j||_Z^p=(||x_i-e_j||_X) ^p+  (||y_i||_Y)^p \geq
 (||x_i||_X)^p - \varepsilon^p+(1 -\varepsilon)^p+ (||y_i||_Y)^p$.
Since $\varepsilon$ is arbitrary, this proves that $T(Z) \geq 2^{1/p}$, so the equality.
\end{proof}

\section {Further remarks and open questions}

The core of the strange behaviour of $T(\cdot)$ is the following
result:

\begin{lemma}\label{core} Every Banach space $X$ can be embedded as a $1$-complemented
hyperplane in a space $Y$ with $T(Y)=1$.
\end{lemma}
\begin{proof} Set $Y=X\oplus_{\infty} \R$ and  consider  in $Y$ the points
$\pm y_0=(0, \pm 1)$. Clearly  $||\pm y_0 ||=1$  and,
for $y=(x,c) \in B_Y$, we have $ ||y \pm y_0||=\max\{ ||x||, |c\pm
1| \}$, and so \, $\min\{ ||y-y_0||,  ||y+y_0||\} \leq 1$.
\end{proof}

Thus, while $T(\ell_\infty)=1$, $T(L^{\infty} [0,1])=2$ since
$T(C(K))=2$ whenever $K$ is an infinite compact Hausdorff space
without isolated points (see \cite{W}) and thus $\ell_\infty$ can
be renormed to have $T(\cdot)=2$. This also follows from the
following result proved in  \cite[Thm. 1.2]{KSW}:  \emph{A space
$Y$ admits a renorming with $T(Y)=2$ if and only if it contains an
isomorphic copy of $\ell_1$}. Which also means that there is a
renorming of $Y=\ell_1 \oplus_ {\infty} \R$ for which $T(Y)=2$.
Since $T(X)<2$ for every reflexive space, no renorming of $Y=X
\oplus \R$ with $T(Y)=2$ exists when $X$ is reflexive.\\

Recall that a Banach space $X$ is said to be \emph{polyhedral} if
the unit ball of any two-dimensional subspace is a polyhedron.
Obviously, $c_0$ is polyhedral  and $ T(c_0)=1$. Moreover, every
subspace of a polyhedral space contains almost-isometric copies of
$c_0$. Nevertheless, there are polyhedral renormings of $c_0$ with
$T(\cdot)$ as close to $2$ as desired (it cannot be $2$ by the
comments above). Consider the following renorming of $c_0$: for
$k\in \N$ set
$$\|(x_n)_n\|_k=  \max _k \left \{ \frac{1}{k} \sum_{j=1}^k|x_{n_j}| \right\}$$
where the maximum is taken over all choices of $k$  different
indexes $n_1, . . . , n_k$. It is easy to check that this space is
polyhedral (see \cite[p.873]{DP}). Moreover, given a finite net
from its unit sphere, let $j$ be an index such that every element
of the net have all components in modulus less than $\varepsilon$
from $j$ onwards. We see that the distance from $k e_j$ to all
elements of the net is at least $ (k-\varepsilon+ k-1)/k$; thus, $T(X)
\geq \frac{2k-1}{k}$. Since $k$ can be as large as we like, $T(X)$
can approach $2$ as much as one wants.\\

An interesting class of Banach spaces with $1<T(X)<2$ is formed by
the \emph{uniformly nonsquare} (UNS is short) spaces. Recall that
a Banach space $X$ is said to be (UNS), if $\sup \big \{  \min
\{||x-y||, ||x+y|| \}: \, x, y \in S_X \big\}<2.$ If a space is
(UNS), then $1< T(X) <2$ (see \cite[Cor. 5.4 and Thm. 5.10]{MP1}).
Next example shows that the converse fails.\\

\noindent  \textbf{Example} The space $X=\R\oplus_1 \ell_p \;
(1\leq p <\infty)$ is not (UNS); we want to show that
$T(X)=2^{1/p}$. By Theorem 2 (1), $T(X) \leq2^{1/p}$; now take a
finite  net in $S_X$ and $\varepsilon>0$. Let the modulus of the
$j$-th  component, for the part in $\ell_p$, be smaller than or
equal to $\varepsilon$ for all elements in the net. Assume that an
element of the net $(c_i, x_i)$ has $||x_i||=b$, so $|c_i|=1-b$;
for $z=(0, e_j)$, the distance from it is at least
$1-b+(b^p-\varepsilon^p+(1-\varepsilon)^p)^{1/p}$,
 so $T(X) \geq 1-b+(b^p+1)^{1/p}$. In $\R^2$, for any $x$ we have
$\|x\|_p/\|x\|_1 \geq 2^{1/p -1}$; so, by taking $x=(1,b)$,  we
see that $(b^p+1)^{1/p} \geq (b+1)( \frac{2^{1/p}}{2})$. An easy
computation then shows that  $T(X) \geq 2^{1/p}$.\\

The equalities $T(\ell_p)=T(L_p)=2^{1/p}$ for $1\leq p\leq \infty$
suggest that spaces with the same ``isometric local structure"
--whatever this may mean-- have the same thickness. A trying
question posed in \cite{CP} is whether $T(X)=T(X^{**})$.

{}

\end{document}